\documentclass{amsart}

\def\R{{\mathbb {R}}}

\def\F{{\mathcal{F}}}

\newtheorem{teo}{Theorem}
\newtheorem{lema}{Lemma}

\theoremstyle{remark}
\newtheorem{remark}{Remark}

\begin{document}

\title[Multiple solutions with critical growth]{Multiple
solutions for the $p-$laplace operator with critical growth}

\author[P. De N\'apoli, J. Fern\'andez Bonder \& A. Silva]
{Pablo L. De N\'apoli, Juli\'an Fern\'andez Bonder \& Anal\'{\i}a Silva}

\address{Pablo De N\'apoli\hfill\break\indent
Departamento  de Matem\'atica, FCEyN \hfill\break\indent UBA (1428) Buenos Aires, Argentina.} 
\email{{\tt pdenapo@dm.uba.ar}\hfill\break\indent 
{\it Web-page: {\tt http://mate.dm.uba.ar/\~{}pdenapo}}}

\address{Juli\'an Fern\'andez Bonder\hfill\break\indent
Departamento  de Matem\'atica, FCEyN \hfill\break\indent UBA (1428) Buenos Aires, Argentina.} 
\email{{\tt jfbonder@dm.uba.ar}\hfill\break\indent 
{\it Web-page: {\tt http://mate.dm.uba.ar/\~{}jfbonder}}}

\address{Anal\'ia Silva\hfill\break\indent
Departamento  de Matem\'atica, FCEyN \hfill\break\indent UBA (1428) Buenos Aires, Argentina.} 
\email{{\tt asilva@dm.uba.ar}}

\thanks{Supported by Universidad de Buenos Aires under grant X078 and X837,
by ANPCyT PICT No. 2006-290 and CONICET (Argentina) PIP 5477 and PIP 5478/1438.
J. Fern\'andez Bonder and Pablo De N\'apoli are members of CONICET. Anal\'ia Silva is a fellow of CONICET.}


\begin{abstract}
In this paper we show the existence of at least three nontrivial solutions to the following quasilinear elliptic equation $-\Delta_p u = |u|^{p^*-2}u +
\lambda f(x,u)$ in a smooth bounded domain $\Omega$ of $\R^N$ with homogeneous Dirichlet boundary conditions on $\partial\Omega$, where $p^*=Np/(N-p)$ is the critical Sobolev exponent and $\Delta_p u = \mbox{div}(|\nabla u|^{p-2}\nabla u)$ is the $p-$laplacian. The proof is based on variational arguments and the classical concentration compactness method.
\end{abstract}

\maketitle

\section{Introduction.}

Let us consider the following nonlinear elliptic problem:
\begin{equation}\tag{P}\label{1.1}
\begin{cases}
-\Delta_p u = |u|^{p^*-2}u + \lambda f(x,u) & \mbox{in } \Omega\\
u = 0 & \mbox{on } \partial\Omega,
\end{cases}
\end{equation}
where $\Omega$ is a bounded smooth domain in $\R^N$, $\Delta_p u =
\mbox{div}(|\nabla u|^{p-2}\nabla u)$ is the $p-$laplacian, $1<p<N$,
$p^*=Np/(N-p)$ is the critical exponent in the Sobolev embedding and $\lambda$
is a positive parameter.

Problems like \eqref{1.1} appears naturally in several branches of pure and applied mathematics, such as  the theory of quasiregular and quasiconformal mappings in Riemannian manifolds (see \cite{E, To}), non-Newtonian fluids, reaction diffusion problems, flow through porous media, nonlinear elasticity, glaciology, etc. (see \cite{AD, AEK, ACh, D}).

The purpose of this paper, is to prove the existence of at least three nontrivial solutions for \eqref{1.1} under adequate assumptions on the source term $f$ and the parameter $\lambda$. This result extends an old paper by
Struwe \cite{St}. A related result for the nonlinear boundary condition case can be found in \cite{FB}.

Here, no oddness condition is imposed in $f$ and a positive, a negative and a sign-changing solution are found (For odd 
nonlinearities it is well known that infinitely many solutions can be obtained in many situations, by using the tools of 
critical point theory, as the $\mathbb{Z}_2$-symmetric version of the mountain pass theorem. See for instance \cite{R} or \cite{DJM}). 

The proof of our result relies on the variational principle of Ekeland (see \cite{S}) complemented with the, by now, well known concentration compactness method of P.L. Lions (see \cite{Lions}). 

One of the advantages in using Ekeland's variational principle is that it allows to split the geometry of the problem from the compactness aspect of it. This approach simplifies the one applied by M. Struwe  in \cite{St} for the subcritical case.

The use of the concentration compactness method to deal with the $p-$Laplacian has been used by several authors before. One of the first results in this direction was obtained by J. Garc\'{\i}a Azorero and I. Peral in \cite{GAP}. Here, we borrow some ideas from that work.

For related results with subcritical growth, see the already mentioned \cite{St} and more recently \cite{BL, ZCL}.

Also, for related results with critical growth, but with subcritical power perturbation, see the seminal work of M. Guedda and L. 
Veron \cite{GV} and, more recently, the paper by S. Cingolani and G. Vannella \cite{CV}.

Throughout this work, by (weak) solutions of \eqref{1.1} we understand critical points of the associated energy functional acting on the Sobolev space $W_0^{1,p}(\Omega)$:
\begin{equation}\label{Phi}
\Phi(v)=\frac1p\int_{\Omega} |\nabla v|^p \, dx - \frac{1}{p^*}\int_{\Omega}
|u|^{p^*}\, dx - \lambda \int_{\Omega} F(x,v)\, dx,
\end{equation}
where $F(x,u) = \int_0^u f(x,z)\, dz$.

We will denote
\begin{equation}\label{FyG}
\F_\lambda(v)=\int_{\Omega}\frac{1}{p^*}|v|^{p^*} + \lambda F(x,v)\, dx,
\end{equation}
so the functional $\Phi$ can be rewritten as
$$
\Phi(v) = \frac1p \|v\|_{W^{1,p}(\Omega)}^p - \F_\lambda(v).
$$


\section{Assumptions and statement of the results.}

The precise assumptions on the source terms $f$ are as
follows:

\begin{enumerate}
\item[(F1)] $f:\Omega\times\R\to\R$, is a measurable function with respect to the first argument and continuously differentiable with respect to the second argument for almost every $x\in\Omega$. Moreover, $f(x,0)=0$ for every $x\in\Omega$.


\item[(F2)] There exist constants $c_1\in (0,1/(p^*-1))$, $k_2\in(p,p^*)$, $0<c_3<c_4$, such that for any $u\in L^q(\Omega)$ and $p<q<p^*$.
\begin{align*}
c_3\|u\|^q_{L^q(\Omega)}&\leq k_2\int_\Omega F(x,u)\,dx\leq\int_\Omega f(x,u)u\,dx\\
&\leq c_1\int_\Omega f_u(x,u)u^2\,dx\leq c_4\|u\|^q_{L^q(\Omega)}
\end{align*}
\end{enumerate}

\begin{remark}
Observe that this set of hypotheses on the nonlinear term $f$ are weaker than the ones considered by \cite{S}.
\end{remark}

\begin{remark}
We exhibit now two families of examples of nonlinearities that fulfill all of our hypotheses.
\begin{itemize}
\item $f(x,u)=|u|^{q-2}u+|u_+|^{r-2}u_+$, if $r\leq q$.

Hypotheses (F1)--(F2) are clearly satisfied.

\item $f(x,u)=\begin{cases}
  |u|^{q-2}u+|u|^{r-2}u & u\geq0\\
  |u|^{q-2}u+|u|^{r-2}u & u<0\\
  \end {cases}$, if $r\leq q$.

Hypotheses (F1)--(F2) are, again, clearly satisfied.
\end{itemize}
\end{remark}

So the main result of the paper reads:

\begin{teo} Under assumptions {\em (F1)--(F2)}, there
exist $\lambda^*>0$ depending only on $n, p, q$ and the constant $c_3$ in {\em (F2)}, such that for every $\lambda>\lambda^*$, there exists three different, nontrivial, (weak) solutions of problem \eqref{1.1}. Moreover these solutions are, one positive, one negative and the other one has non-constant sign.
\end{teo}

\section{Proof of Theorem 1.}

The proof uses the same approach as in \cite{St}. That is, we will construct three disjoint sets $K_i\ne\emptyset$ not containing $0$ such that $\Phi$ has a critical point in $K_i$. These sets will be subsets of $C^{1}-$manifolds $M_i\subset W^{1,p}(\Omega)$ that will be constructed by imposing a sign restriction and a normalizing condition.

In fact, let
\begin{align*}
& M_1 = \left\{u \in W^{1,p}_0(\Omega)\colon \int_{\Omega} u_+>0 \mbox{ and }
\int_{\Omega}|\nabla u_+|^{p}-|u_+|^{p^*}dx=\int_{\Omega}\lambda f(x,u) u_+
dx\right\},\\
& M_2 = \left\{u \in W^{1,p}_0(\Omega)\colon \int_\Omega u_->0 \mbox{ and }
\int_\Omega|\nabla u_-|^{p}-|u_-|^{p^*}dx=\int_\Omega\lambda f(x,u) u_- dx\right\},\\
& M_3 = M_1\cap M_2.
\end{align*}
where $u_+ = \max\{u,0\}$, $u_-=\max\{-u,0\}$ are the positive and negative parts of $u$, and $\langle\cdot,\cdot\rangle$ is the duality pairing of $W^{1,p}(\Omega)$.

Finally we define
\begin{align*}
& K_1 = \{ u\in M_1\ |\ u\ge 0\},\\
& K_2 = \{ u\in M_2\ |\ u\le 0\},\\
& K_3 = M_3.
\end{align*}

First, we need a Lemma to show that these sets are nonempty and, moreover, give some properties that will be useful in the proof of the result.

\begin{lema}\label{tlambda}
For every $w_0\in W^{1,p}_0(\Omega)$, $w_0>0$ ($w_0<0$), there exists $t_\lambda>0$ such that $t_\lambda w_0\in M_1 (\in M_2)$. Moreover, $\lim_{\lambda\to\infty} t_\lambda = 0$.

As a consequence, given $w_0, w_1\in W^{1,p}_0(\Omega)$, $w_0>0$, $w_1<0$, with disjoint supports, there exists $\bar t_\lambda, \underbar t_\lambda>0$ such that $\bar t_\lambda w_0 + \underbar t_\lambda w_1 \in M_3$. Moreover $\bar t_\lambda, \underbar t_\lambda \to 0$ as $\lambda\to\infty$.
\end{lema}

\begin{proof}
We prove the lemma for $M_1$, the other cases being similar.

For $w\in W^{1,p}_0(\Omega)$, $w \geq 0$, we consider the functional
$$
\varphi_1(w)= \int_{\Omega} |\nabla w|^p -|w|^{p^*} - \lambda f(x,w)w \, dx.
$$
Given $w_0>0$, in order to prove the lemma, we must show that $\varphi_1(t_\lambda w_0)=0$ for some $t_\lambda>0$. Using hypothesis (H3), we have that:
$$
\varphi_1(tw_0) \geq A t^p - B t^{p^*} - \lambda c_4 C t^q
$$
and
$$
\varphi_1(tw_0) \leq A t^p - B t^{p^*} - \lambda c_3 C t^q,
$$
where the coefficients $A$, $B$ and $C$ are given by:
$$
A = \int_{\Omega} |\nabla w_0|^p \,dx, \quad B= \int_{\Omega} |w_0|^{p^*} \,dx,
\quad C = \int_{\Omega} |w_0|^q \, dx.
$$
Since $p<q<p^*$ it follows that $\varphi_1(tw_0)$ is positive for $t$ small enough, and negative for $t$ big enough. Hence, by Bolzano's theorem, there exists some $t=t_\lambda$ such that $\varphi_1(t\lambda u)=0$. (This
$t_\lambda$ needs not to be unique, but this does not matter for our purposes).

In order to give an upper bound for $t_\lambda$, it is enough to find some $t_1$, such that $\varphi_1(t_1 w_0)<0$. We observe that:
$$
\varphi_1(t w_0) \leq A t^p -  \lambda c_3 C t^q.
$$
so it is enough to choose $t_1$ such that $A t_1^p - \lambda c_3 C t_1^q=0$, i.e.,
$$
t_1 = \left(\frac{A}{c_3 \lambda C}\right)^{1/(q-p)}.
$$
Hence, again by Bolzano's theorem, we can choose $t_\lambda \in [0,t_1]$, which
implies that $t_\lambda \to 0$ as $\lambda \to +\infty$.
\end{proof}

For the proof of the Theorem, we need also the following Lemmas.

\begin{lema}\label{lema1}
There exist $c_j>0$ such that, for every $u\in K_i$, $i=1,2,3$,
$$
\int_\Omega |\nabla u|^{p}\, dx \leq C_1 \left(\lambda \int_\Omega f(x,u)u\, dx
+ \int_\Omega |u|^{p^*}\, dx\right) \leq C_2 \Phi(u) \leq C_3 \left(\int_\Omega
|\nabla u|^{p}\, dx\right).
$$
\end{lema}

\begin{proof}
As $u\in K_i$, we have that
$$
\int_\Omega|\nabla u|^{p}dx=\int_\Omega\lambda f(x,u)u+|u|^{p^*} dx.
$$
This proves the first inequality.

Now, by (F3)
$$
\int_\Omega F(x,u) dx\leq\frac{1}{c_2}\int_\Omega f(x,u)u dx.
$$
Furthermore,
\begin{align*}
\left|\lambda\int_\Omega F(x,u) dx \right|=\lambda\int_\Omega F(x,u)
dx &\leq\frac{1}{c_2}\int_\Omega\lambda f(x,u)u
dx\\
&=\frac{1}{c_2}\left(\int_\Omega|\nabla u|^{p}-|u|^{p^*} dx\right),
\end{align*}
so,
\begin{equation}\label{cota.F}
-\lambda\int_\Omega F(x,u) dx\leq\frac{1}{c_2}\left(\int_\Omega|\nabla
u|^{p}-|u|^{p^*} dx\right).
\end{equation}

By \eqref{cota.F}, we have:
\begin{align*}
\Phi(u) &= \int_{\Omega} \frac{1}{p} |\nabla u|^{p} - \frac{1}{p^*} |u|^{p^*} -
\lambda F(x,u)\, dx \\
&\leq \int_{\Omega} \frac{1}{p} |\nabla u|^{p} - \frac{1}{p^*} |u|^{p^*}\, dx +
\frac{1}{c_2} \left(\int_\Omega |\nabla u|^{p} - |u|^{p^*}\, dx\right)\\
&\leq \left(\frac{1}{c_2} + \frac{1}{p}\right) \int_\Omega |\nabla u|^{p}\, dx.
\end{align*}

This proves the third inequality.

To prove the middle inequality we proceed as follows:
\begin{align*}
\Phi(u)=&\int_{\Omega}\frac{1}{p}|\nabla u|^{p}-\frac{1}{p^*}|u|^{p^*}-\lambda
F(x,u)dx\geq\int_{\Omega}\frac{1}{p}(|\nabla u|^{p}-|u|^{p^*})-\lambda F(x,u)dx\\
\geq&\frac{1}{p}\int_\Omega\lambda f(x,u)u dx-\lambda\int_\Omega F(x,u)
dx\geq(\frac{1}{p}-\frac{1}{c_2})\lambda\int_\Omega f(x,u) dx.
\end{align*}

This finishes the proof.
\end{proof}

\begin{lema}\label{lema2}
There exists $c>0$ such that
\begin{align*}
\|\nabla u_+\|_{L^{^p}(\Omega)} & \geq c\quad\forall\,u\in K_1,\\
\|\nabla u_-\|_{L^{^p}(\Omega)} & \geq c\quad\forall\,u\,\in K_2,\\
\|\nabla u_+\|_{L^{^p}(\Omega)}\, \mbox{,}\,\|\nabla u_-\|_{L^{^p}(\Omega)} &
\geq c\quad\forall\,u\in K_3.
\end{align*}
\end{lema}

\begin{proof}
By the definition of $K_i$,by (F3) and the Poincar\'e inequality we have that
\begin{align*}
\|\nabla u_\pm\|^{p}_{L^{^p}(\Omega)}=&\int_\Omega\lambda f(x,u)u_\pm+|u_\pm|^{p^*} dx\leq C\| u_\pm\|^{q}_{L^{^q}(\Omega)}+\|
u_\pm\|^{p^*}_{L^{p^*}(\Omega)}\\
\leq &c_1\| \nabla u_\pm\|^{q}_{L^{^p}(\Omega)}+ c_2\| \nabla
u_\pm\|^{p^*}_{L^{^p}(\Omega)}.
\end{align*}
As $p<q<p^*$, this finishes the proof.
\end{proof}

\begin{lema}\label{lema3}
There exists $c>0$ such that $\Phi(u)\ge
c\|\nabla u\|_{L^p(\Omega)}^p$ for every $u\in W^{1,p}(\Omega)$ if $\|u\|_{W_0^{1,p}(\Omega)}$ is small enough.
\end{lema}

\begin{proof}
By (F3) and the Poincar\'e inequality we have
\begin{align*}
\Phi(u)&=\int_{\Omega}\frac{1}{p}|\nabla u|^{p}-\frac{1}{p^*}|u|^{p^*}-\lambda F(x,u)dx\\
&\geq\frac{1}{p}\|\nabla u\|^{p}_{L^{p}(\Omega)}-\frac{1}{p^*}\|u\|^{p^*}_{L^{p^*}(\Omega)}-C\|u\|^{q}_{L^{^q}(\Omega)}\\
&\geq\frac{1}{p}\|\nabla u\|^{p}_{L^{p}(\Omega)}- C_1(\|\nabla u\|^{p^*}_{L^{^p}(\Omega)}+\|\nabla u\|^{q}_{L^{^p}(\Omega)})\\
&\geq C\|\nabla u\|_{L^p(\Omega)},
\end{align*}
if $\|\nabla u\|_{L^p(\Omega)}$ is small enough, as $p<q<p^*$.
\end{proof}

The following lemma describes the properties of the manifolds $M_i$.

\begin{lema}\label{lema4}
$M_i$ is a $C^{1}$ sub-manifold of $W_0^{1,p}(\Omega)$ of co-dimension 1 $(i=1,2)$, 2 $(i=3)$ respectively. The sets $K_i$ are complete. Moreover, for every $u\in M_i$ we have the direct decomposition
$$
T_u W_0^{1,p}(\Omega) = T_u M_i \oplus \mbox{span}\{u_+, u_-\},
$$
where $T_u M$ is the tangent space at $u$ of the Banach manifold $M$. Finally, the projection onto the first component in this decomposition is uniformly continuous on bounded sets of $M_i$.
\end{lema}

\begin{proof}
Let us denote
\begin{align*}
&\bar M_1 = \left\{u\in W^{1,p}_0(\Omega)\colon \int_\Omega u_+\, dx > 0
\right\},\\
&\bar M_2 = \left\{u\in W^{1,p}_0(\Omega)\colon \int_\Omega u_-\, dx > 0
\right\},\\
&\bar M_3 = \bar M_1 \cap \bar M_2.
\end{align*}
Observe that $M_i\subset \bar M_i$.

The set $\bar M_i$ is open in $W^{1,p}(\Omega)$,
therefore it is enough to prove that $M_i$ is a $C^1$ sub-manifold of $\bar M_i$. In order to do this, we will construct a $C^{1}$ function $\varphi_i:\bar M_i\to \R^d$ with $d=1\ (i=1,2)$, $d=2\ (i=3)$ respectively and
$M_i$ will be the inverse image of a regular value of $\varphi_i$.

In fact, we define: For $u\in \bar M_1,$
$$
\varphi_1(u)=\int_{\Omega}|\nabla u_+|^{p}-|u_+|^{p^*}-\lambda f(x,u) u_+ \,
dx.
$$
For $u\in \bar M_2,$
$$
\varphi_2(u)=\int_{\Omega}|\nabla u_-|^{p}-|u_-|^{p^*}-\lambda f(x,u) u_-\, dx.
$$
For $u\in \bar M_3,$
$$
\varphi_3(u)=(\varphi_1(u),\varphi_2(u)).
$$
Obviously, we have $M_i = \varphi_i^{-1}(0)$. From standard arguments (see \cite{DJM}, or the appendix of \cite{R}), $\varphi_i$ is of class $C^1$. Therefore, we only need to show that $0$ is a regular value for $\varphi_i$. To this end we compute, for $u\in M_1$,
\begin{align*}
\langle\nabla\varphi_1(u),u_+\rangle&= p \|\nabla u_+\|^p-p^*\|u_+\|^{p^*}-\lambda\int_\Omega f(x,u)u_+- f_u(x,u)u_+^2\,dx\\
&\leq p^*\left(\|\nabla u_+\|^p-\|u_+\|^{p^*}\right)-\lambda\int_\Omega f(x,u)u_+- f_u(x,u)u_+^2\,dx\\
&\leq (p^*\lambda-\lambda)\int_\Omega f(x,u)u_+\,dx-\int_\Omega
f_u(x,u)u_+^2\,dx.
\end{align*}
By (F3) the last term is bounded by
\begin{align*}
(p^*\lambda-\lambda-\frac{\lambda}{c_1})\int_\Omega f(x,u)u_+\,dx
&=\left(p^*-1-\frac{1}{c_1}\right)\left(\|\nabla u_+\|^p_{L^p(\Omega)}-\|u_+\|^{p^*}_{L^{p^*}(\Omega)}\right)\\
&\leq\left(p^*-1-\frac{1}{c_1}\right)\|\nabla u_+\|^p_{L^p(\Omega)}.
\end{align*}
Recall that $c_1<1/(p^*-1)$. Now, the last term is strictly negative by Lemma \ref{lema2}. Therefore, $M_1$ is a $C^{1}$ sub-manifold of $W^{1,p}(\Omega)$.
The exact same argument applies to $M_2$. Since trivially
$$
\langle \nabla \varphi_1(u), u_-\rangle = \langle \nabla \varphi_2(u),
u_+\rangle =0
$$
for $u\in M_3$, the same conclusion holds for $M_3$.

To see that $K_i$ is complete, let $u_k$ be a Cauchy sequence in $K_i$, then $u_k\to u$ in $W^{1,p}(\Omega)$. Moreover, $(u_k)_{\pm}\to u_{\pm}$ in $W^{1,p}(\Omega)$. Now it is easy to see, by Lemma \ref{lema2} and by continuity that $u\in K_i$.

Finally, by the first part of the proof we have the decomposition 
$$
T_u W^{1,p}(\Omega) = T_u M_i \oplus \mbox{span}\{u_+\}
$$
Where $M_1=\{u:\varphi_1(u)=0\}$ and $T_u M_1=\{v:\langle\nabla \varphi_1(u),v\rangle=0\}$. Now let $v\in T_u W_0^{1,p}(\Omega)$ be a unit tangential vector, then $v = v_1 + v_2$ where $v_2=\alpha u_+$ and $v_1=v-v_2$. Let us take $\alpha$ as
$$
\alpha = \frac{\langle \nabla\varphi_1(u),v\rangle}{\langle
\nabla\varphi_1(u),u_+\rangle}.
$$
With this choice, we have that $v_1\in T_u M_1$. Now
$$
\langle\varphi_1(u),v_1\rangle=0.
$$
The very same argument to show that $T_uW^{1,p}(\Omega) = T_u M_2\oplus\langle u_-\rangle$ and $T_uW^{1,p}(\Omega) = T_u M_3\oplus \langle u_+,u_-\rangle$.

From these formulas and from the estimates given in the first part of the proof, the uniform continuity of the projections onto $T_u M_i$ follows.
\end{proof}

Now, we need to check the Palais-Smale condition for the functional $\Phi$ restricted to the manifold $M_i$. To this end, we need the following lemma from \cite{GAP} which proves the Palais-Smale condition for the unrestricted functional below certain energy level.

\begin{lema}[J. Garc\'{\i}a-Azorero, I. Peral, \cite{GAP}]\label{gap}
Let $S_p$ be the best Sobolev constant
$$
S_p := \inf_{\phi\in C^{\infty}_c(\R^N)} \frac{\int_{\R^N} |\nabla \phi|^p\,
dx}{\left(\int_{\R^N} |\phi|^{p^*}\, dx\right)^{p/p^*}}.
$$
Then, the unrestricted functional $\Phi$ verifies the Palais-Smale condition for energy level $c$ for every $c<\frac{1}{N}S_p^{N/p}$. \label{PSc}
\end{lema}

The proof of Lema \ref{gap} is based on the concentration compactness method.

Now, we can prove the Palais-Smale condition for the restricted functional.

\begin{lema}\label{lema5}
The functional $\Phi|_{K_i}$ satisfies the Palais-Smale condition for energy level $c$ for every $c<\frac{1}{N}S_p^{N/p}$.
\end{lema}

\begin{proof}
Let $\{u_k\}\subset K_i$ be a Palais-Smale sequence, that is $\Phi(u_k)$ is uniformly bounded and $\nabla \Phi|_{K_i}(u_k)\to 0$ strongly. We need to show that there exists a subsequence $u_{k_j}$ that converges strongly in $K_i$.

Let $v_j\in T_{u_j}W_0^{1,p}(\Omega)$ be a unit tangential vector such that
$$
\langle \nabla \Phi(u_j), v_j\rangle = \|\nabla
\Phi(u_j)\|_{W^{-1,p'}(\Omega)}.
$$
Now, by Lemma \ref{lema4}, $v_j = w_j + z_j$ with $w_j\in
T_{u_j}M_i$ and $z_j\in \mbox{span}\{(u_j)_+, (u_j)_-\}$.

Since $\Phi(u_j)$ is uniformly bounded, by Lemma \ref{lema1}, $u_j$ is uniformly bounded in $W_0^{1,p}(\Omega)$ and hence $w_j$ is uniformly bounded
in $W_0^{1,p}(\Omega)$. Therefore
$$
\|\nabla\Phi(u_j)\|_{W^{-1,p'}(\Omega)} = \langle \nabla \Phi(u_j), v_j\rangle
= \langle \nabla \Phi|_{K_i}(u_j), v_j\rangle\to 0.
$$

As $w_j$ is uniformly bounded and $\nabla \Phi|_{K_i}(u_k)\to 0$ strongly, the inequality converges strongly to 0. Now the result follows by Lema \ref{gap}.
\end{proof}

We now immediately obtain

\begin{lema}\label{lema6}
Let $u\in K_i$ be a critical point of the restricted functional $\Phi|_{K_i}$. Then $u$ is also a critical point of the unrestricted functional $\Phi$ and hence a weak solution to \eqref{1.1}.
\end{lema}

With all this preparatives, the proof of the Theorem follows easily.
\begin{proof}[\bf Proof of Theorem 1]
To prove the Theorem,  we need to check that the functional $\Phi\mid_{K_i}$ verifies the hypotheses of the Ekeland's Variational Principle \cite{Ekeland}.

The fact that $\Phi$ is bounded below over $K_i$ is a direct consequence of the construction of the manifold $K_i$.

Then, by Ekeland's Variational Principle, there existe $v_k\in K_i$ such that
$$
\Phi(v_k)\to c_i:=\inf_{K_i}\Phi\qquad \mbox{ and }\qquad
(\Phi\mid_{K_i})'(v_k)\to 0.
$$
We have to check that if we choose $\lambda$ large, we have that $c_i<\frac{1}{N}S_p^{N/p}$. This follows easily from Lemma \ref{tlambda}. For instance, for $c_1$, we have that choosing $w_0\geq 0$,
$$
c_1 \leq \Phi(t_\lambda w_0) \leq \frac{1}{p} t_{\lambda}^p \int_\Omega |\nabla w_0|^p
\, dx
$$ 
Hence $c_1 \to 0$ as $\lambda \to 0$. Moreover, it follows from the estimate of $t_\lambda$ in Lemma \ref{tlambda}, that $c_i<\frac{1}{N}S_p^{N/p}$ for $\lambda>\lambda^*(p,q,n,c_3)$. The other cases are similar.

From Lemma \ref{PSc}, it follows that $v_k$ has a convergent subsequence, that we still call $v_k$. Therefore $\Phi$ has a critical point in $K_i$, $i=1,2,3$
and, by construction, one of them is positive, other is negative and the last one changes sign.
\end{proof}

\end{document}